\documentclass[11pt,pdflatex]{amsart}

\usepackage{amssymb}
\usepackage{amsmath}
\usepackage{amsfonts}
\usepackage{hyperref}
\usepackage{multirow, comment}
\usepackage{mathrsfs}
\usepackage[mathscr]{euscript}

\theoremstyle{plain}
\newtheorem{theorem}{Theorem}[section]

\newtheorem{corollary}[theorem]{Corollary}
\newtheorem{lemma}[theorem]{Lemma}

\theoremstyle{definition}

\newtheorem{definition}[theorem]{Def\/inition}

\newcommand{\epf}{{\ifhmode\unskip\nobreak\hfil\penalty50 \hskip1em
\else\nobreak\fi \nobreak\mbox{}\hfil\mbox{$\square$} \parfillskip=0pt
\finalhyphendemerits=0 \par\vskip5pt}}

\newcommand{\lra}{\longrightarrow}

\newcommand{\CC}{\mathbf{C}}
\newcommand{\FF}{\mathbf{F}}

\newcommand{\QQ}{\mathbf{Q}}
\newcommand{\RR}{\mathbf{R}}
\newcommand{\ZZ}{\mathbf{Z}}

\newcommand{\BBB}{\mathbb{B}}

\newcommand{\Hasse}{\mathrm{Hasse}}

\newcommand{\kk}{\mathbf{k}}
\newcommand{\hh}{\mathbf{h}}
\newcommand{\ee}{\mathbf{e}}

\newcommand{\CO}{\mathcal{O}}

\newcommand{\gn}{\mathfrak{n}}

\newcommand{\gp}{\mathfrak{p}}

\newcommand{\gI}{\mathfrak{I}}

\newcommand{\uhp}{\mathfrak{H}}
\newcommand{\Fpbar}{{\overline{\FF}_p}}
\usepackage[usenames]{color}
\usepackage{color}
\usepackage{amssymb}
\usepackage{graphicx, epsfig}
\usepackage{latexsym, amsfonts, amscd, amsmath}
\usepackage{mathrsfs}
\makeindex 
\input xy
\xyoption{all}

\definecolor{orange}{rgb}{1,0.5,0}
\definecolor{Indigo}{rgb}{0.2,0.1,0.7}
\definecolor{Violet}{rgb}{0.5,0.1,0.7}

\theoremstyle{definition}

\theoremstyle{remark}

\newenvironment{mat}{\left(\begin{smallmatrix}
\end{smallmatrix}\right)}{enddef}

\numberwithin{equation}{subsection}
\numberwithin{figure}{subsection} \numberwithin{table}{subsection}

\newcommand{\End}{{\operatorname{End}}}
\newcommand{\Fr}{{\operatorname{Fr }}}
\newcommand{\Hom}{{\operatorname{Hom}}}

\newcommand{\Ker}{{\operatorname{Ker}}}

\newcommand{\Spec}{{\operatorname{Spec }}}

\newcommand{\Ver}{{\operatorname{Ver }}}

\newcommand{\SL}{{\operatorname{SL }}}

\newcommand{\Lie}{{\operatorname{Lie}}}

\newcommand{\gern}{{\frak{n}}}

\newcommand{\gerp}{{\frak{p}}}

\newcommand{\gerI}{{\frak{I}}}

\newcommand{\uA}{{\underline{A}}}
\newcommand{\uB}{{\underline{B}}}
\newcommand{\uC}{{\underline{C}}}

\newcommand{\calA}{{\mathcal{A}}}

\newcommand{\calF}{{\mathcal{F}}}

\newcommand{\calH}{{\mathcal{H}}}

\newcommand{\calO}{{\mathcal{O}}}
\newcommand{\calP}{{\mathcal{P}}}

\def\BB{\mathbb{B}}
\def\CC{\mathbb{C}}
\def\DD{\mathbb{D}}

\def\FF{\mathbb{F}}

\def\PP{\mathbb{P}}
\def\QQ{\mathbb{Q}}
\def\RR{\mathbb{R}}

\def\ZZ{\mathbb{Z}}

\newcommand{\ra}{{\; \rightarrow \;}}

\newcommand{\ol}{{\mathcal{O}_L}}

\newcommand{\Cl}{{\text{\rm Cl}}}

\newcommand{\Xbar}{\overline{X}}

\newcommand{\cAbar}{\overline{\calA}}

\newcommand{\Qpbar}{\overline{\QQ}_p}

\begin{document}

\title[Weights of Hilbert modular forms]{Minimal weights of Hilbert modular forms in characteristic $p$}

\author{Fred Diamond}
\email{fred.diamond@kcl.ac.uk}
\address{Department of Mathematics,
King's College London, WC2R 2LS, UK}
\author{Payman Kassaei}
\email{payman.kassaei@kcl.ac.uk}
\address{Department of Mathematics,
King's College London, WC2R 2LS, UK}

\thanks{F.D.~was partially supported by EPSRC Grant EP/L025302/1.}
\date{December 2016}

\begin{abstract}  
We consider mod $p$ Hilbert modular forms associated to a totally real field of degree $d$ in which $p$ is
unramified.  We prove that every such form arises by multiplication by partial Hasse invariants
from one whose weight (a $d$-tuple of integers) lies in a certain cone contained in the set of non-negative
weights, answering a question of Andreatta and Goren.  The proof is based on properties of the 
Goren--Oort stratification on mod $p$ Hilbert modular varieties established by Goren and Oort, and
Tian and Xiao.
\end{abstract}

\maketitle

\section{Introduction}  \label{sec:intro}

Let $F$ be a totally real field of degree $d = [F:\QQ]>1$.    
In this paper, we consider mod-$p$ Hilbert modular forms associated to $F$.
Recall that the weight of a (geometric) Hilbert modular form associated to $F$ is a $d$-tuple of integers.
Unlike the case of Hilbert modular forms in characteristic zero, or indeed mod-$p$ classical modular forms
(i.e., the excluded case $d=1$), there are non-zero Hilbert modular forms in characteristic $p$
for which some components of the weight are negative.
Indeed the partial Hasse invariants considered by Andreatta and Goren in \cite{AG} are examples of such
forms, and they raise the question (\cite[15.8]{AG}) of whether all such forms are obtained from forms
of non-negative weight by multiplication by partial Hasse invariants.  We answer the
question here in the affirmative, assuming $p$ is unramified in $F$.   In fact, we prove a stronger result
along these lines (Corollary~\ref{cor:filtration} below), which also immediately implies a vanishing
result (Corollary~\ref{cor:vanishing}) recently proven independently by
Goldring and Koskivirta using methods different from ours (see \cite{GK}
 and their forthcoming work with Stroh on related questions).

We now explain this in more detail and give the precise statements of our results.
Let $\CO_F$ denote the ring of integers of $F$, and fix non-zero ideals $\gn$ and
$\gI$ of $\CO_F$.  We let $p$ be a prime which is relatively prime to $\gn$ and the
discriminant $D_{F/\QQ}$.  Assuming that $\gn$ is sufficiently small (e.g., contained
in $N\CO_F$ for some integer $N > 3$), then there is a 
scheme $X = X^\gI(\gn)$, smooth of dimension $d$ over $\ZZ_p$, which
parametrizes $d$-dimensional abelian varieties with $\CO_F$-action,
together with suitable polarization data (depending on $\gI$) and level structure
(depending on $\gn$); see \S\ref{section: HMV} below.  The scheme $X$ is a model for a disjoint union of 
Hilbert modular varieties over $\CC$ with complex points $\Gamma_i \backslash \uhp^d$,
where $\{\Gamma_i = \Gamma_1^{\gI}(\gn)_i\}$ is a certain collection of congruence subgroup of $\SL_2(F)$
acting on $d$ copies of the complex upper-half plane $\uhp$ via the $d$ embeddings $F \to \RR$.
In this paper, we are only concerned with a geometric special fibre $\Xbar:=X_{\Fpbar}$.

Let $\BBB$ denote the set of embeddings $F \to \Qpbar$; since $p$ is unramified
in $F$, we may identify $\BBB$ with the set of ring homomorphisms $\CO_F \to \Fpbar$.
For $\beta \in \BBB$, we let $\ee_\beta$ denote the basis element of $\ZZ^\BBB$
associated to $\beta$.

Let $\omega = s_*\Omega^1_{\cAbar/\Xbar}$ denote the sheaf of invariant differentials of the
universal abelian variety $\pi:\cAbar \to \Xbar$.  Then $\omega$ is, locally on $\Xbar$, free of rank one
over $\CO_{\Xbar} \otimes \CO_F \cong \oplus_{\beta\in \BBB}\ \CO_{\Xbar}$.
Therefore, $\omega$ decomposes as a direct sum over $\beta \in \BBB$
of line bundles $\omega_\beta$.  For any $\kk = \sum k_\beta \ee_\beta \in \ZZ^\BBB$, we let 
$\omega^\kk = \otimes_\beta\ \omega_\beta^{k_\beta}$, and we define the space
of $\gerI$-polarized geometric Hilbert modular forms of level $\gn$, and weight $\kk$ over $\Fpbar$
to be the set of global sections:
$$M^\gI_\kk(\gn,\Fpbar) = H^0(\Xbar, \omega^{\kk}).$$

Note that we may decompose $\BBB  = \coprod \BBB_{\gp}$, where the union
is over the set of primes $\mathfrak{p}$ of $\CO_F$ dividing $p$, and $\BBB_{\gp}
= \{\, \beta:\CO_F \to \Fpbar\,|\,\ker(\beta) = \gp\,\}$.    Let $\sigma$ denote the
Frobenius automorphism $x \mapsto x^p$ of $\Fpbar$; note that composition with
$\sigma$ defines a cyclic permutation of each $\BBB_\gp$, of order $|\BBB_\gp|
 = [F_\gp:\QQ_p]$.  The partial Hasse invariant $H^\gerI_\beta$, defined in \cite[7.12]{AG},
 is an element of $M^\gI_{\hh_\beta}(\gn,\Fpbar)$, i.e., a mod-$p$ Hilbert modular form of
 weight $\hh_\beta$, where $\hh_\beta = p \ee_{\sigma^{-1}\circ\beta} - \ee_\beta$.
 Note that $\sigma^{-1}\circ \beta = \beta$ if and only if $\beta \in \BBB_\gp$ for
 a prime $\gp$ of degree one, in which case $\hh_\beta = (p-1) \ee_\beta$; otherwise
 the coefficient in $\hh_\beta$ of $\ee_\beta$ (resp.~$\ee_{\sigma^{-1}\circ\beta}$) is $-1$
 (resp.~$p$).   For any $\kk \in \ZZ^\BBB$, multiplication by $H^\gerI_\beta$ defines an injection
$M^\gI_{\kk-\hh_\beta}(\gn,\Fpbar) \hookrightarrow M^\gI_{\kk}(\gn, \Fpbar)$.
The key result is the following:
\begin{theorem}  \label{thm:key}  Suppose that $\kk = \sum k_\beta\ee_\beta \in \ZZ^\BBB$, and that
$\beta \in \BBB$ is such that $p k_\beta < k_{\sigma^{-1}\circ\beta}$.  Then multiplication by the
partial Hasse invariant $H^\gerI_\beta$ induces an isomorphism:
$$M^\gI_{\kk-\hh_\beta}(\gn,\Fpbar) \,\,\,\stackrel{\sim}{\lra} \,\,\, M^\gI_{\kk}(\gn, \Fpbar).$$
\end{theorem}
The proof of the theorem is based on properties of the Goren--Oort stratification established by
Goren and Oort~\cite{GO} and Tian and Xiao~\cite{TX}.

We will apply Theorem~\ref{thm:key} to deduce a property of the weight filtration on Hilbert modular
forms.  To define the weight filtration, Andreatta and Goren prove (\cite[8.19]{AG}) that if $f$ is a
non-zero element of $M^\gI_{\kk}(\gn, \Fpbar)$, then there is a unique maximal element of the set
$$\left\{\left.\, \sum_\beta n_\beta \ee_\beta \in \ZZ_{\ge 0}^\BBB\, \,\right|\, \,  f = g\prod_\beta (H_\beta^{\gI})^{n_\beta}\,\mbox{for some $g \in M^\gI_{\kk - \sum n_\beta \hh_\beta}(\gn, \Fpbar)$}\,\right\}$$
(under the usual partial ordering).  The weight filtration $\Phi(f)$ is then defined to be
$\kk - \sum n_\beta \hh_\beta$ (i.e., the weight of $g$), where $\sum n_\beta \ee_\beta$
is this maximal element.

We now state the consequences of Theorem~\ref{thm:key} in terms of two cones in $\QQ^\BBB$.
We first define the {\em minimal, standard,} and {\em Hasse} cones as follows
$$C^{\min}  = \left\{\left.\, \sum_\beta x_\beta \ee_\beta \in \QQ^\BBB\, \,\right|\, \, p x_\beta \ge x_{\sigma^{-1}\circ\beta}\,\ \mbox{for all $\beta\in \BBB$}\,\right\},$$
$$C^{\rm st}=\left\{\left.\, \sum_\beta x_\beta \ee_\beta \in \QQ^\BBB\, \,\right|\, \,   x_\beta \ge 0\,\ \mbox{for all $\beta\in \BBB$}\,\right\},$$

$$C^{\Hasse}  = \left\{\left.\, \sum_\beta y_\beta \hh_\beta \in \QQ^\BBB \,\,\right|\, \, y_\beta \ge 0 \,\ \mbox{for all $\beta\in \BBB$}\,\right\}.$$
Note that we have the containments $C^{\min} \subset  C^{{\rm st}} \subset C^{\Hasse}$, each of which is an equality if and only if $p$ splits completely in $F$.

\begin{corollary}  \label{cor:filtration}  If $f$ is a non-zero form in $M^\gI_\kk(\gn,\Fpbar)$ for some
$\kk \in \ZZ^\BBB$, then $\Phi(f) \in C^{\min}$.
\end{corollary}

The motivation for the result arises from the first author's work with Shu Sasaki, in the course
of which we observed that an analogous property seems to hold for labelled weights of crystalline
lifts of two-dimensional mod $p$ local Galois representations.  As discussed above, this answers
the question raised by Andreatta and Goren in \cite[15.8]{AG}; it also immediately implies
the following result which has recently been independently proven by Goldring and Koskivirta
(Corollary~3.2.4 of \cite{GK}):

\begin{corollary} \label{cor:vanishing} If $M^\gI_\kk(\gn,\Fpbar) \neq 0$, then $\kk \in C^{\Hasse}$.
\end{corollary}

\subsection*{Acknowledgements}  We are grateful to Shu Sasaki for conversations on research setting
the stage for this paper, to Wushi Goldring for informing us of his work with Koskivirta and Stroh, and to
the referee for helpful remarks.

\section{Hilbert modular varieties}\label{section: HMV}

 Let $F$ be a totally real field of degree $d>1$, and $\calO_F$ its ring of integers. Let $p$ be a prime number which is unramified in $F$, and $\gern$ an ideal of $\calO_F$ relatively prime to $p$. We also assume that $\gern$ is contained in $N\calO_F$ for some $N>3$.     For a  prime ideal $\gerp$  of $\calO_F$ dividing $p$, let $\kappa_\gerp = \calO_F/\gerp$, a finite field of order $p^{f_\gerp}$. Let $\kappa$ be a finite field containing an isomorphic copy of each $\kappa_\gerp$. We denote by $\mathbb{Q}_\kappa$ the fraction field of $W(\kappa)$, and fix, once and for all,  embeddings $\mathbb{Q}_\kappa \hookrightarrow \QQ_p^{\rm ur}
 \subset \Qpbar$ (inducing $\kappa \hookrightarrow \Fpbar$), as well as embeddings $\overline{\QQ} \hookrightarrow \overline{\QQ}_p $ and
 $\overline{\QQ} \hookrightarrow \CC$.

Let $\BB={\rm Emb}(F,\mathbb{Q}_\kappa)$.  We write $\BB=\textstyle\coprod_{\mathfrak{p}|p}
\mathbb{B}_\mathfrak{p}$, where 
\[
\mathbb{B}_\mathfrak{p}= \{\beta\in\mathbb{B}\colon
\beta^{-1}(pW(\kappa))=\gerp\},
\]
for a prime ideal $\mathfrak{p}$ dividing $p$.
Since $p$ is unramified in $F$, $\BB$ is canonically identified with $\Hom(\calO_F,\kappa)$. Using the fixed embeddings of $\overline{\QQ}$ into  $\overline{\QQ}_p$ and $\CC$, we can also identify  $\BB$ with the set of infinite places of $F$, and can think of $\BB_\gerp$ as the set of infinite places inducing the prime ideal $\gerp$.

Let $\sigma$ denote the Frobenius automorphism of $\mathbb{Q}_\kappa$, lifting $x \mapsto
x^p$ modulo $p$. It acts on $\BB$ by $\beta \mapsto \sigma
\circ \beta$. It acts transitively on each $\BB_\mathfrak{p}$. The decomposition \[\mathcal{O}_F
\otimes_\mathbb{Z} W(\kappa)=\bigoplus_{\beta \in \mathbb{B}}
W(\kappa)_\beta,\] where $W(\kappa)_\beta$ is $W(\kappa)$ with the
$\mathcal{O}_F$-action given via the embedding $\beta$, induces a decomposition
\[\Lambda=\bigoplus_{\beta\in \mathbb{B}} \Lambda_\beta,\] for any $\mathcal{O}_F
\otimes_\mathbb{Z} W(\kappa)$-module $\Lambda$. 

\

Let $A$ be an abelian scheme over a scheme $S$, with
real multiplication  defined by $\iota\colon \calO_F \rightarrow \End_S(A)$. Then
the dual abelian scheme $A^\vee$ has an induced canonical real
multiplication.  We define $\calP_A = \Hom_{\calO_F}(A, A^\vee)^{\rm
sym}$, which we view as an \'etale sheaf on $S$ of projective $\calO_F$-modules of rank $1$ with
a notion of positivity in which the positive elements correspond to
$\calO_F$-equivariant polarizations.

Fix  a representative $[(\gerI,
\gerI^+)]\in \Cl^+(F)$. Consider the functor
 \[
 \calH: \langle{\rm Schemes}/\ZZ_p \rangle \ra \langle{\rm Sets}\rangle
 \] that associates to a $\ZZ_p$-scheme $S$, the set of all isomorphism classes of  $\underline{A}/S=(A/S,\iota,\lambda,\alpha)$ such that:  
 \begin{itemize}
 \item $A$ is an abelian scheme of relative dimension $d$ over  $S$;
\item $\iota\colon\mathcal{O}_F \hookrightarrow {\rm End}_S(A)$ is a ring
homomorphism endowing $A$ with real multiplication by  $\calO_F$;
 \item   the map $\lambda$ is an $\calO_F$-linear isomorphism $\lambda\colon (\calP_A, \calP_A^+)
\rightarrow (\gerI, \gerI^+)$  of sheaves on the \'etale site of $S$ such that  $A \otimes_\ol \gerI \cong
A^\vee$;
\item  the map $\alpha\colon (\calO_F/\gern)^2 \ra A[\gern]$ is an $\calO_F$-linear isomorphism of group schemes. 
\end{itemize}
Since $p$ is unramified in $F$, the existence of $\lambda$ is equivalent to $\Lie(A)$ being a locally free $(\calO_F\otimes\calO_S)$-module of rank one. 

There is a scheme $X^\gerI(\gern)=X$  of relative dimension $d$ over $\Spec(\ZZ_p)$ which represents the above functor. It is called the Hilbert modular scheme of polarization $\gerI$ and level $\Gamma(\gern)$, and is smooth over $\Spec(\ZZ_p)$.

We let $X_\kappa:=X \otimes_{\ZZ_p} \kappa$. It represents the above functor  restricted to schemes over $\kappa$.  We denote the base extension of $X_\kappa$ under our fixed embedding $\kappa \subset \Fpbar$ by $\Xbar$.

Let $\calA/X$ (similarly, $\calA_\kappa/X_\kappa$, $\cAbar/\Xbar$) denote the universal abelian scheme, and denote by $\pi$ the structural morphism in all cases. Let $\omega_{\calA/X}= \pi_*\Omega^1_{\calA/X}$ denote the sheaf of invariant differentials. Similarly define $\omega_{\calA_\kappa/X_\kappa}$ and $\omega_{\cAbar/\Xbar}$. When there is no possibility of confusion, we will simply use $\omega$ to denote any of these sheaves.

Since $p$ is unramified in $F$, we know that  $\omega_{\cAbar/\Xbar}$ is locally free of rank one over  $\calO_{\Xbar} \otimes_\ZZ \calO_F \cong \oplus_{\beta\in \BBB}\  \CO_{\Xbar}$.
Therefore, we have 
\[
\omega=\oplus_{\beta\in \BBB}\  \omega_{\beta},
\]
where, for each $\beta \in \BBB$, the sheaf $\omega_\beta$ is locally free of rank one over $\calO_{\Xbar}$. 

For $\beta \in \BBB$, we let $\ee_\beta$ denote the basis element of $\ZZ^\BBB$
associated to $\beta$. For any $\kk = \sum k_\beta\ee_\beta \in \ZZ^\BBB$, we define
\[
\omega^\kk=\otimes_{\beta \in \BB}\  \omega_\beta^{k_\beta}.
\]

\section{Hilbert modular forms in characteristic $p$.}

Let notation be as in \S \ref{section: HMV}. Let $R$ by any $\kappa$-algebra. We define the space of $\gerI$-polarized geometric Hilbert modular forms over $R$ of weight $\kk = \sum k_\beta\ee_\beta \in \ZZ^\BB$, and level $\Gamma(\gern)$ to be
 \[
M^\gI_{\kk}(\gn,R)=H^0(X_\kappa \otimes R, \omega^\kk).
\] 
 
 Important examples of such modular forms are the partial Hasse invariants defined first in \cite{G}, and studied further in \cite{AG}. Let $\Ver: \calA^{(p)}_\kappa \ra \calA_\kappa$ be the Verschiebung morphism of the universal abelian variety over $X_\kappa$. The induced morphism $\Ver^*: \omega_{\calA_\kappa/X_\kappa} \ra \omega_{\calA^{(p)}_\kappa/X_\kappa}$ can be written as $\Ver=\oplus_{\beta \in \BB}\ \Ver_\beta^*$, where
\[
\Ver_\beta^*:  \omega_\beta=(\omega_{\calA_\kappa/X_\kappa})_\beta  \ra (\omega_{\calA^{(p)}_\kappa/X_\kappa})_\beta \cong \omega^p_{\sigma^{-1}\circ\beta}
\]
gives a section of $\Hom(\omega_\beta,\omega^p_{\sigma^{-1}\circ\beta})\cong\omega_\beta^{-1}\otimes \omega_{\sigma^{-1}\circ\beta}^p$,  which we denote by $H_\beta^\gerI$.  Thus $H_\beta^\gerI$ is a $\gerI$-polarized geometric Hilbert modular forms over $\kappa$ of weight 
 \[
 \hh_\beta: = p\ee_{\sigma^{-1}\circ\beta}-\ee_\beta,
 \]
 and level $\Gamma(\gern)$.

While we haven't discussed $q$-expansions in this article, we would like to point out that each $H_\beta^\gerI$ has $q$-expansions equal to $1$ at an unramified cusp on each connected component, which motivates the following question: given a mod-$p$ modular form $f$, what is the maximum power of each $H^\gerI_\beta$ by which $f$ is divisible?  In \cite[8.19]{AG}, it is proven that if $f$ is a
non-zero element of $M^\gI_{\kk}(\gn, \Fpbar)$, then there is a unique maximal element of the set
$$\left\{\left.\, \sum_{\beta\in\BB} n_\beta \ee_\beta \in \ZZ_{\ge 0}^\BBB\, \,\right|\, \,  f = g\prod_{\beta\in\BB} (H^{\gerI}_\beta)^{n_\beta}\,\mbox{for some $g \in M^\gI_{\kk - \sum n_\beta \hh_\beta}(\gn, \Fpbar)$}\,\right\}$$
(under the usual partial ordering).  The weight filtration $\Phi(f)$ is then defined to be
$\kk - \sum n_\beta \hh_\beta$ (i.e., the weight of $g$), where $\sum n_\beta \ee_\beta$
is this maximal element.
 
\section{The Goren--Oort Stratification}\label{section: GO}
In this section, we recall the definition of the Goren--Oort stratification on $X_\kappa$ and some of its properties proven in \cite{GO}.  We will then recall some key results proven by Tian--Xiao \cite{TX} regarding the global structure of these strata, which we will use in the proof of our main theorem.

 Let $x=(A,\iota,\lambda,\alpha)$  be a closed point of $X_\kappa$.  The type of $x$ is defined by
\begin{equation} \tau(x) = \{\beta \in \BB: \DD\left(\Ker(\Fr_A)\cap
\Ker(\Ver_A)\right)_\beta\neq 0\},
\end{equation}where $\DD$ denotes the contravariant Dieudonn\'e module.

Let $T \subseteq \BB$.  There is a locally closed subset
$W_T$ of $X_\kappa$ with the property that a closed point $x$ of
$X_\kappa$  belongs to $W_T$ if and only if
$\tau(x) = T$. 
We define $Z_T$ to be the Zariski closure of $W_T$.

The following properties of this stratification are proven in \cite{GO}.
\begin{itemize}
\item The collection $\{W_T: T \subseteq \BB\}$ is a
stratification of $X_\kappa$; we have  $Z_T=\bigcup_{T' \supseteq T} W_{T'}$.
\item Each $W_T$ is non-empty, and equi-dimensional of dimension $d - |T|$.
\item Each $Z_T$ is nonsingular.

\item For any $T \subset \BB$, we have $Z_T=\bigcap_{\beta\in T} Z(H^\gerI_\beta)$ scheme theoretically, where $Z(.)$ denotes the vanishing locus.
\item $W_\BB=Z_\BB$ is the non-empty finite set of all points $x=(A,\iota,\lambda,\alpha)$ where $A$ is a superspecial abelian variety.
 
\end{itemize}

\begin{lemma}\label{lemma: intersect}  Let $T \subset \BB$. Every connected component of $Z_T$ intersects $Z_\BB$ nontrivially.

\end{lemma}

\begin{proof} This is Corollary 2.3.12 of \cite{GO}.
\end{proof}

We now recall some results proven in \cite{TX} about the global structure of the strata $Z_T$.

\begin{theorem}\label{theorem: TX}(Tian--Xiao) Let $\gerp|p$ be a prime ideal of $\calO_F$ such that  $f_\gerp>1$. Let $\beta_0 \in \BB_\gerp$. The Goren--Oort stratum $Z_{\{\beta_0\}}$ is isomorphic to a $\PP^1$-bundle over the quaternionic Shimura variety $X^B_K$ associated to $B^\times$, where $B$ is the quaternion algebra over $F$ that is ramified exactly at $\{\sigma^{-1}\circ\beta_0,\beta_0\}$, and $K$ is a level structure maximal at $p$.
\end{theorem}

A much more general version of this result, which is based on an idea of Helm~\cite{H},  is proven in \cite{TX} (see Theorem~5.2), where such a recipe is given for $Z_T$ for an arbitrary subset $T$ of $\BB$. 

\begin{lemma} \label{lemma: torsion}Let $\gerp|p$ be a prime ideal of $\calO_F$, and $\beta_0 \in \BB_\gerp$. Let $T \subset \BB$ be such that $\beta_0 \in T$. We have the following isomorphism of line bundles on $Z_T$:
\[
\omega_{\beta_0}|_{Z_T} \cong \omega^{-p}_{\sigma^{-1}\circ\beta_0}|_{Z_T}
\]
In particular, if $\BB_\gerp \subset T$, then for all $\beta \in \BB_\gerp$ we have $(\omega_\beta|_{Z_T})^{p^{f_\gerp}}\cong \calO_{Z_T}$.
\begin{proof} Clearly, it is enough to prove this for $T=\{\beta_0\}$. Recall that $\calA$ denotes the universal abelian scheme over $\Xbar$. There is a short exact sequence 
\[
0 \ra \omega_{\calA/\Xbar} \ra H^1_{dR} (\calA/\Xbar) \ra \omega^{\vee}_{\calA^\vee/\Xbar} \ra 0.
\]
Isolating the $\beta_0$-component and applying the isomorphism $\omega_{\calA/\Xbar} \cong \omega_{\calA^\vee/\Xbar}$ (induced by a choice of prime-to-$p$ polarization on the universal abelian variety), we deduce that $\wedge^2 H^1_{dR} (\calA/\Xbar)_{\beta_0}$ is the trivial line bundle on $\Xbar$.

Now, we consider the Verschiebung morphism 
\[
\Ver_{\beta_0}:  H^1_{dR} (\calA/\Xbar)_{\beta_0} \ra H^1_{dR} (\calA^{(p)}/\Xbar)_{\beta_0} 
\]
 which has image equal to $(\omega^{(p)})_{\beta_0}=\omega^p_{\sigma^{-1}\circ \beta_0}$. By definition, the vanishing of $H^\gerI_{\beta_0}$ is equivalent to $\omega_{\beta_0}$ equalling the kernel of $\Ver_{\beta_0}$. In particular, on $Z_{\{\beta_0\}}$, we have  a short exact sequence of sheaves
\[
0 \ra \omega_{\beta_0} \ra H^1_{dR} (\calA/\Xbar)_{\beta_0} \ra  \omega^p_{\sigma^{-1}\circ \beta_0}  \ra 0.
\]
This implies that on $Z_{\{\beta_0\}}$ we have $\calO_{Z_{\beta_0}} \cong \wedge^2 H^1_{dR} (\calA/\Xbar)_{\beta_0} \cong \omega_{\beta_0} \otimes \omega_{\sigma^{-1}\circ\beta_0}^p$, as claimed.


The last claim follows by applying the above result for all $\beta \in \BB_\gerp$.

\end{proof}

\end{lemma}

\begin{lemma}\label{lemma: fibres} Let $\beta_0 \in \BB_\gerp$ be such that $f_\gerp>1$. By Theorem \ref{theorem: TX}, the Goren--Oort stratum $Z_{\{\beta_0\}}\subset \Xbar$ is isomorphic to a $\PP^1$-bundle over a quaternionic Shimura variety $X^B_K$. Let $x$ be a geometric point of $X^B_K$, and denote by $\PP^1_x$ the corresponding fibre of the $\PP^1$-bundle. Then, for $\beta\in\BB$, we have the following equalities:
 $$\omega_\beta|_{\PP^1_x} \cong \begin{cases} 
   \calO_{\PP^1_x}(p)& \mbox{if $\beta=\beta_0$,} \\
  \calO_{\PP^1_x}(-1) & \mbox{if $\beta=\sigma^{-1}\circ\beta_0$,}\\
 \calO_{\PP^1_x} & \mbox{otherwise.}
  \end{cases}$$
\end{lemma}

\begin{proof} This is contained in \cite{TX}, but we sketch the proof for the convenience of the reader.  The construction of the isomorphism between  $Z_{\{\beta_0\}}$  and the  $\PP^1$-bundle over $X^B_K$  is given in \cite[5.11-5.24]{TX}. It is first shown that the connected components of Hilbert modular varieties and quaternionic Shimura varieties are naturally isomorphic with connected components of certain unitary Shimura varieties (respecting the Goren-Oort stratifications). Hence, it is enough to prove the above statement over appropriate unitary Shimura varieties. For the rest of this proof, we will use notation directly from \cite{TX}. To be more precise, it is enough to prove the statement with $Z_{\{\beta_0\}} \subset \Xbar$  replaced with ${\rm Sh}_{K'}({G'_\emptyset})_{ k_0,\{   \tilde{\beta}_0 \} } \subset  {\rm Sh}_{K'}({G'_\emptyset})_{k_0}$,  $X^B_K$ replaced with ${\rm Sh}_{K'}({G'_{\{\sigma^{-1}\circ\tilde{\beta}_0,\tilde{\beta}_0 \}})}_{k_0}$, and $\omega_\beta$ replaced by $\omega^\circ_{{\bf{\rm A}}^\vee/S,\tilde{\beta}}$. The fibre of the $\PP^1$-bundle over ${\rm Sh}_{K'}({G'_{\{\sigma^{-1}\circ\tilde{\beta}_0,\tilde{\beta}_0 \}})}_{k_0}$ at a point $\uB/S$ on ${\rm Sh}_{K'}({G'_{\{\sigma^{-1}\circ\tilde{\beta}_0,\tilde{\beta}_0 \}})}_{k_0}$ (where $\uB$ denotes an abelian scheme $B/S$ with PEL structure) is the projectivization of $H^{\rm dR}_1(B/S)^\circ_{\sigma^{-1}\circ\tilde{\beta}_0}$. In \cite[\S 5.15]{TX}, for any point $\uA/S$ on ${\rm Sh}_{K'}({G'_\emptyset})_{k_0,\{\tilde{\beta}_0\}}$, abelian schemes $\uC$ and $\uB$ along with isogenies $\phi_A: A \ra C$, and $\phi_B: B \ra C$ are constructed such that $\uB/S$ is a point on  ${\rm Sh}_{K'}({G'_{\{\sigma^{-1}\circ\tilde{\beta}_0,\tilde{\beta}_0 \}})}_{k_0}$. This data satisfies several properties including that $\phi_{A,*,\sigma^{-1}\circ\tilde{\beta}_0}$ and $\phi_{B,*,\sigma^{-1}\circ\tilde{\beta}_0}$ are isomorphisms. The isomorphism between ${\rm Sh}_{K'}({G'_\emptyset})_{ k_0,\{   \tilde{\beta}_0 \} }$ and the $\PP^1$-bundle over ${\rm Sh}_{K'}({G'_{\{\sigma^{-1}\circ\tilde{\beta}_0,\tilde{\beta}_0 \}})}_{k_0}$ sends $\uA/S$ to $(\uB,J)$, where $J$ is the line in $H^{\rm dR}_1(B/S)^\circ_{\sigma^{-1}\circ\tilde{\beta}_0}$ given by $J=\phi_{B,*,\sigma^{-1}\circ\tilde{\beta}_0}^{-1}\phi_{A,*,\sigma^{-1}\circ\tilde{\beta}_0}(\omega^\circ_{A^\vee/S,\sigma^{-1}\circ \tilde{\beta_0}})$. This immediately implies that the restriction of $\omega^\circ_{A^\vee/S,\sigma^{-1}\circ \tilde{\beta_0}}$ to each fibre of the $\PP^1$-bundle is isomorphic to $\calO(-1)$ on $\PP^1$. Applying Lemma \ref{lemma: torsion}, it follows that  the restriction of $\omega_{\beta_0}$ to each fibre of the corresponding $\PP^1$-bundle is isomorphic to $\calO(p)$.

To prove the statement for $\beta\not\in \{\sigma^{-1}\circ\beta_0,\beta_0\}$, it is enough to note that, by construction (see proof of Lemma 5.18 in \cite{TX}), $H^{\rm dR}_1(B/S)^\circ_{\tilde{\beta}}$ is canonically isomorphic to  $H^{\rm dR}_1(A/S)^\circ_{\tilde{\beta}}$ since $\tilde{\beta} \not\in \tilde{\Delta}(\{\beta_0\})^+ \cup  \tilde{\Delta}(\{\beta_0\})^-=\{\tilde{\beta_0},\tilde{\beta_0}^c\}$. This shows that  $\omega^\circ_{A^\vee/S,\tilde{\beta}}$ depends canonically on $\uB$, that  is, the isomorphic image of $\omega^\circ_{{\bf{\rm A}}^\vee/S,\tilde{\beta}}$   on the $\PP^1$-bundle is a pullback from the base, and hence has trivial fibres.

\end{proof}

\section{The Main Theorem}

In this section, we prove our main theorem and discuss some of its consequences.

\begin{theorem}  \label{thm:key}  Suppose that $\kk\! =\! \sum k_\beta\ee_\beta \in \ZZ^\BBB$, and that
$\beta_0\in\BBB$ is such that $p k_{\beta_0} < k_{\sigma^{-1}\circ{\beta_0}}$ for some $\beta_0 \in \BBB$.  Then, multiplication by the
partial Hasse invariant $H^\gerI_{\beta_0}$ induces an isomorphism:
$$M^\gI_{\kk-\hh_{\beta_0}}(\gn,\Fpbar) \,\,\,\stackrel{\sim}{\lra} \,\,\, M^\gI_{\kk}(\gn, \Fpbar).$$
If $f_\gerp=1$, then a stronger result holds: we have $M^\gI_{\kk}(\gn, \Fpbar)=0$ for all $\kk$ for which $k_{\beta_0}<0$.
\end{theorem}

\begin{proof}  Let $\gerp$ be a prime ideal of $\calO_F$ such that $\beta_0 \in \BB_\gerp$. We distinguish two cases.

\

Case(1): $f_\gerp>1$.  Consider the short exact sequence 
\[
0 \ra \omega^{\kk-\hh_{\beta_0}} \ra \omega^\kk \ra \calF \ra 0
\]
of sheaves on $\Xbar$, where the first map is multiplication by $H^\gerI_{\beta_0}$. To prove the result, it  is enough to show that $H^0(\Xbar,\calF)$ vanishes. Clearly, $\calF$ is supported on the vanishing locus of $H^\gerI_{\beta_0}$, i.e., $Z_{\{\beta_0\}}$, and we have $H^0(\Xbar,\calF)=H^0(Z_{\{\beta_0\}}, \omega^\kk)$. We show that the latter vanishes.  

By Theorem \ref{theorem: TX}, the closed Goren--Oort stratum $Z_{\{\beta_0\}}$ is isomorphic to a $\PP^1$-bundle over a quaternionic Shimura variety $X^B_K$. Let $s \in H^0(Z_{\{\beta_0\}}, \omega^\kk)$.   Let $x$ be a geometric point of $X^B_K$, and denote by $\PP^1_x$ the corresponding fibre of the $\PP^1$-bundle. By Lemma \ref{lemma: fibres}, we have 
\[
\omega^{\kk}|_{\PP^1_x}\cong \calO(p k_{\beta_0} - k_{\sigma^{-1}\circ{\beta_0}}).
\]
By assumption $p k_{\beta_0} -k_{\sigma^{-1}\circ{\beta_0}} <0$, so it follows that  $H^0(\PP^1_x,\omega^\kk|_{\PP^1_x})=0$. This shows that $s|_{\PP^1_x}=0$ for all geometric points $x$ of $X^B_K$. Hence $s=0$ as claimed.

\

Case(2): $f_\gerp=1$. Here the assumption becomes $k_{\beta_0}<0$. We prove the stronger assertion that 
\[
H^0(\Xbar,\omega^\kk)=0,
\]
if $k_{\beta_0}<0$. This will imply $M^\gI_{\kk-\hh_{\beta_0}}(\gn,\Fpbar)=M^\gI_{\kk}(\gn, \Fpbar)=0$.

Write $\BB=\{\beta_0,\beta_1,...,\beta_{d-1}\}$, and, for $0 \leq i \leq d-1$, let  $T_i:=\BB-\{\beta_0,...,\beta_i\}$.  We prove by  induction on $i$ that 
\[
H^0(Z_{T_i},\omega^\kk)=0,
\]
for all $0 \leq i \leq d-1$. 

We first prove the result for $i=0$.  By properties of the Goren--Oort stratification recalled in \S \ref{section: GO}, the closed stratum $Z_{T_{0}}$ is equi-dimensional of dimension $1$.  Let $W$ be an arbitrary connected component of $Z_{T_0}$. By Lemma \ref{lemma: intersect}, $W$ has a superspecial point on it. In particular, $H^\gerI_{\beta_0}$ has a nonempty vanishing 
locus on $W$. This  implies that $\deg(\omega_{\beta_0}|_{W})>0$ on every connected component $W$ of $Z_{T_{0}}$. By the second statement in Lemma \ref{lemma: torsion}, we find that for all $1\leq i \leq d-1$, the line bundle $\omega_{\beta_i}$ restricts to a torsion line bundle on $Z_{T_0}$. Hence, on every connected component $W$ of  $Z_{T_0}$ we have 
 \[
 \deg(\omega^\kk|_{W})=k_{\beta_0}\deg(\omega_{\beta_0}|_{W})<0.
 \]
 Since every such connected component is projective, this proves that $H^0(Z_{T_{0}},\omega^\kk)=0$. 
 
 Now we prove the induction step. Assume that $1 \le i \le d-1$, and that  the result holds for $i -1$. We prove that it holds also for $i$. Consider the short exact sequence 
 \[
 0 \ra \omega^{\kk-\hh_{\beta_i}}|_{Z_{T_{i}}}\ra \omega^\kk|_{Z_{T_{i}}} \ra \calF \ra 0
 \]
 of sheaves on $Z_{T_{i}}$, where the first map is multiplication by $H^\gerI_{\beta_i}|_{Z_{T_{i}}}$. The sheaf $\calF$ is supported on the vanishing locus of $H^\gerI_{\beta_i}$ on $Z_{T_{i}}$, which is $Z_{T_{i-1}}$, and we have $H^0(Z_{T_{i}},\calF)=H^0(Z_{T_{i-1}}, \omega^\kk)$=0, by the induction assumption. This shows that multiplication by $H^\gerI_{\beta_i}|_{Z_{T_{i}}}$ induces an isomorphism
  \[
 H^0(Z_{T_{i}}, \omega^{\kk-\hh_{\beta_i}}) \,\,\,\stackrel{\sim}{\lra} \,\,\, H^0({Z_{T_{i}}},\omega^\kk).
 \]
 Since the $\beta_0$-component of $\kk-\hh_{\beta_i}$ is negative, we can repeat the argument  any number of times and deduce that for all $N>0$, multiplication by  $(H^\gerI_{\beta_i})^N|_{Z_{T_{i}}}$ induces an isomorphism
 \[
 H^0(Z_{T_{i}}, \omega^{\kk-N\hh_{\beta_i}}) \,\,\,\stackrel{\sim}{\lra} \,\,\, H^0({Z_{T_{i}}},\omega^\kk).
 \]
 By Lemma \ref{lemma: intersect}, $H^\gerI_{\beta_i}$ has a nonempty vanishing locus on any connected component of $Z_{T_{i}}$. This, along with the fact that every connected component of $Z_{T_{i}}$ is irreducible, implies that any nonzero section of $\omega^\kk$ on $Z_{T_{i}}$ is divisible by at most a finite power of $H^\gerI_{\beta_i}|_{Z_{T_{i}}}$. We therefore conclude that $H^0({Z_{T_{i}}},\omega^\kk)=0$. This completes the proof by induction. In particular, for $i=d-1$,  we get 
 \[
 H^0(Z_{T_{d-1}},\omega^\kk)=H^0(\Xbar,\omega^\kk)=0.
 \]
 
 This ends the proof of our main theorem.

\end{proof}

We note the following geometric interpretation of Theorem~\ref{thm:key}:  if $\kk\! =\! \sum k_\beta\ee_\beta \in \ZZ^\BBB$ is such that
$p k_{\beta_0} < k_{\sigma^{-1}\circ{\beta_0}}$, then every Hilbert modular form of weight $\kk$ vanishes on $Z_{\{\beta_0\}}$.

We now prove some interesting consequences of Theorem~\ref{thm:key}. First, we define three cones in $\QQ^\BB$.

\begin{definition}

We define the {\em minimal, standard,} and {\em Hasse} cones as follows
$$C^{\min}  = \left\{\left.\, \sum_\beta x_\beta \ee_\beta \in \QQ^\BBB\, \,\right|\, \, p x_\beta \ge x_{\sigma^{-1}\circ\beta}\,\ \mbox{for all $\beta\in \BBB$}\,\right\},$$
$$C^{\rm st}=\left\{\left.\, \sum_\beta x_\beta \ee_\beta \in \QQ^\BBB\, \,\right|\, \,   x_\beta \ge 0\,\ \mbox{for all $\beta\in \BBB$}\,\right\},$$

$$C^{\Hasse}  = \left\{\left.\, \sum_\beta y_\beta \hh_\beta \in \QQ^\BBB \,\,\right|\, \, y_\beta \ge 0 \,\ \mbox{for all $\beta\in \BBB$}\,\right\}.$$

\end{definition}
It is easily seen that we always have containments  
$$C^{\min} \subset C^{\rm st} \subset C^{\Hasse},$$ each of which is an equality if and only if $p$ splits completely in $F$.

In \cite[15.8]{AG}, Andreatta and Goren ask the following question: Over the complex numbers, a non-zero Hilbert modular form has weight in $C^{\rm st}$  . In characteristic $p$ this is no longer true as the example of the partial Hasse invariant shows. Note though that the filtration of a partial Hasse invariant is the trivial character  (weight $(0,...,0)$) that lies in $C^{\rm st}$. We therefore ask: is there an example of a modular form whose filtration is not in $C^{\rm st}$? is not in $C^{\Hasse}$?

As a corollary to our main theorem, we can answer that there are no such examples, and in fact the following stronger statement is true:

\begin{corollary}  \label{corollary: filtration}  If $f$ is a non-zero form in $M^\gI_\kk(\gn,\Fpbar)$ for some
$\kk \in \ZZ^\BBB$, then $\Phi(f) \in C^{\min}$.
\end{corollary}

\begin{proof} This follows immediately from our main theorem. Let $\Phi(f)=\sum_\beta w_\beta \ee_\beta$ be the filtration of $f$.  Then there are integers $n_\beta\geq 0$, and $g \in M^\gI_{\Phi(f)}(\gn, \Fpbar)$, where 
\[
f = g\prod_{\beta\in\BB} {(H^\gerI_\beta)}^{n_\beta},
\]
and  such that $g$ is not divisible by any partial Hasse invariant $H^\gerI_\beta$.  Hence,  by Theorem \ref{thm:key}, we must have $p w_{\beta} \geq w_{\sigma^{-1}\circ{\beta}}$ for all $\beta \in \BB$, namely, $\Phi(f)=\sum_\beta w_\beta \ee_\beta \in C^{\min}$.
\end{proof}

\

It follows that the phenomenon of negative weights for mod-$p$ Hilbert modular forms  is controlled by the weights of partial Hasse invariants. More precisely we deduce the following result, recently independently proven by  
Goldring and Koskivirta (Corollary~3.2.4 of \cite{GK}):

\begin{corollary} \label{corollary: vanishing} If $M^\gI_\kk(\gn,\Fpbar) \neq 0$, then $\kk \in C^{\Hasse}$.
\end{corollary}

\begin{proof} Assume $f$ is a nonzero form in $M^\gI_\kk(\gn,\Fpbar)$. By Corollary \ref{corollary: filtration}, we know that  
\[
\Phi(f)\in C^{\min} \subset C^\Hasse.
\]
By definition of filtration, there are nonnegative integers $n_\beta$ such that  $\kk=\Phi(f)+\sum n_\beta \hh_\beta$, which clearly belongs to $C^\Hasse$.

\end{proof}

\end{document}